\documentclass[11pt]{amsart}
\usepackage[marginratio=1:1,height=614pt,width=400pt,tmargin=117pt]{geometry}
\usepackage{amsmath,amsthm,amssymb}
\usepackage{comment}
\usepackage{enumitem}
\usepackage{cleveref}
\usepackage{mathtools}
\usepackage{csquotes} 
\usepackage[final]{pdfpages}

\newtheorem{theorem}{Theorem}[section]
\newtheorem{lemma}[theorem]{Lemma}

\newtheorem{corollary}[theorem]{Corollary}

\newtheorem*{Hilbert 19th problem}{Hilbert 19th problem}

\theoremstyle{definition}

\theoremstyle{remark}
\newtheorem{remark}[theorem]{Remark}

\numberwithin{equation}{section}

\usepackage{accents}
\newcommand{\ubar}[1]{\underaccent{\bar}{#1}}

\newcommand{\R}{\mathbb{R}}
\newcommand{\N}{\mathbb{N}}

\DeclareMathOperator{\divr}{div}

\DeclareMathOperator{\dist}{dist}

\title[On supercritical divergence-free drifts]
      {On supercritical divergence-free drifts}

\author[Bian Wu]{Bian Wu}
\date{\today}

\begin{document}

\begin{abstract}
For second-order elliptic or parabolic equations with subcritical or critical drifts, it is well-known that the Harnack inequality holds and their bounded weak solutions are H\"older continuous. We construct time-independent supercritical drifts in $L^{n-\lambda}(\R^n)$ with arbitrarily small $\lambda>0$ such that the Harnack inequality and the H\"older continuity fail in both the elliptic and the parabolic cases, thus confirming a conjecture by Seregin, Silvestre, {\v{S}}ver{\'a}k and Zlato{\v{s}} in \cite{seregin2012divergence}. These results are sharp, and they also apply to a toy model of the axi-symmetric Navier-Stokes equations in space dimension $3$.
\end{abstract}

\maketitle

\section{Introduction}

\subsection{De Giorgi-Nash-Moser iteration} In his address at the 1900 International Congress of Mathematicians, Hilbert stated the famous 19th problem: Are the solutions of regular variational problems always necessarily analytic? On the one hand, Bernstein in 1904, and later I. Petrowski, E. Hopf, J. Schauder, C. B. Morrey, and L. Nirenberg showed sufficiently regular solutions of the Euler–Lagrange equations are real analytic. On the other hand, one could apply the variational methods pioneered by Riemann and Hilbert to produce weak solutions of the Euler–Lagrange equations in Sobolev spaces, thus reducing Hilbert 19th problem to showing that weak solutions $v \in H^1$ of the second-order elliptic PDEs
\begin{equation} \label{div_form}
  -\partial_i(a_{ij} \partial_j v) = 0,
\end{equation}
with uniformly elliptic bounded coefficients $\{ a_{ij} \}$ are sufficiently regular. For example, proving the  H\"older continuity of $v$ for some strictly positive H\"older exponent is enough to close the gap. This was done independently by E. De Giorgi \cite{de1957sulla} and J. Nash \cite{nash1958continuity}. De Giorgi’s argument was later simplified by Moser \cite{moser1960new}. The key H\"older estimates have also been generalized to parabolic equations of non-divergence form by Krylov and Safanov \cite{krylov1980certain}, Aronson \cite{aronson1968non} and other mathematicians, and thus have become a standard tool called De Giorgi-Nash-Moser iteration. This breakthrough as well as relevant generalizations allow to prove strong a priori estimates for very general systems. This, however, only partially answers Nash's high expectations on his celebrated regularity result:

\begin{displayquote}
	Little is known about the existence, uniqueness and smoothness of solutions of the general equations of flow for a viscous, compressible and heat conducting fluid. These are a non-linear parabolic system of equations. Also the relationship between this continuum description of a fluid (gas) and the more physically valid statistical mechanical description is not well understood. An interest in these questions led us to undertake this work. It became clear that nothing could be done about the continuum description of general fluid flow without the ability to handle non-linear parabolic equations and that this in turn required an a priori estimate of continuity such as H\"older continuity in general second-order parabolic PDEs. Probably one should first try to prove a conditional existence and uniqueness theorem for the flow equations. This should give existence, smoothness, and unique continuation (in time) of flows, conditional on the non-appearance of certain gross types of singularity, such as infinities of temperature or density. (A gross singularity could arise, for example, from a converging spherical shock wave.) A result of this kind would clarify the turbulence problem.\\
	\hspace*{\fill} --- Page 932, Nash \cite{nash1958continuity}.
\end{displayquote}

\subsection{The parabolic equation of non-divergence form}

If one wants to apply De Giorgi-Nash-Moser iteration in fluid dynamics as Nash noted, one needs to handle the general parabolic analogues of \eqref{div_form} in non-divergence form. As a basic model problem, consider the following Cauchy problem with solenoidal drift $u:\R^n \times [0,\infty) \rightarrow \R^n$
\begin{equation} \label{toy1}
  \begin{cases}
  \partial_t v - \Delta v + (u \cdot \nabla) v = 0, \quad \text{in } \R^n \times [0,\infty) \\
  v(x,0) = v_0(x)
  \end{cases},
\end{equation}
where $v:\R^n \times [0,\infty) \rightarrow \R$ is a scalar-valued function, and its elliptic counterpart. Indeed, this toy model may be derived from linearizations of various equations in fluid dynamics. By adding more constraints between $u$ and $v$, we can recover nonlinear models such as the Navier-Stokes equations, the surface quasi-geostrophic equations and others. The key point that allows to apply De Giorgi-Nash-Moser iteration in these nonlinear problems is to assume the ``non-appearance of certain gross types of singularity'' as in Nash's remark. This is equivalent to assuming certain regularity properties of $u$. However, the rigorous justification of the appearance or the non-appearance of certain singularities in fluids is one of the most challenging problem in analysis. One main goal of this paper is to examine how much regularity on $u$ is needed for De Giorgi-Nash-Moser iteration and what implications this has in fluid equations. \par

To properly classify the regularity requirement on $u$, we use the following natural scaling invariance of \eqref{toy1}: for any $\lambda>0$, if $v$ is a solution of \eqref{toy1} with drift $u$, then $v_\lambda(x,t) := \lambda v(\lambda x, \lambda^2 t)$ is a solution of \eqref{toy1} with drift $u_\lambda(x,t) := \lambda u(\lambda x, \lambda^2 t)$. This scaling invariance also leads to a natural classification of function spaces with respect to this equation: for a function space $B$, we say $B$ is subcritical if $\|u_\lambda\|_B \rightarrow 0$ as $\lambda \rightarrow 0$. $B$ is called critical, if $\|u_\lambda\|_B = \|u\|_B$ for any $\lambda$. $B$ is called supercritical, if $\|u_\lambda\|_B \rightarrow \infty$ as $\lambda \rightarrow 0$. \par

As pointed out by Seregin, Silvestre, {\v{S}}ver{\'a}k and Zlato{\v{s}} in \cite{seregin2012divergence}, if the drift $u$ is divergence-free, one can use Nash's idea in \cite{nash1958continuity} to show the following upper Gaussian bound for the fundamental solution $G$ of \eqref{toy1},
\begin{equation} \label{nash}
  G(x,t;y,s) \leq C(t-s)^{-n/2}.
\end{equation}
Moreover, one can use this Gaussian bound to show that any weak solution $v$ in $L^\infty_tL^2_x \bigcap L^2_tH^1_x$ is locally bounded, but not H\"older continuous. Note that this only requires $u$ to be weakly divergence-free and locally integrable. Here, weak solutions of \eqref{toy1} are distributional solutions in $L^\infty_tL^2_x \bigcap L^2_tH^1_x$. As we shall see later, closing this gap from local boundedness to H\"older continuity would clarify the regularity problem of the axi-symmetric Navier Stokes equations in space dimension 3.

If we further assume the drifts to be subcritical, for example, $u \in L^p(\R^n)$ with $p>n$, Aronson \cite{aronson1968non} in 1968 used De Giorgi-Nash-Moser iteration and showed that weak solutions of the parabolic equation \eqref{toy1} have many nice properties similar to parabolic equations in divergence form, including upper and lower Gaussian bounds for the fundamental solutions, Harnack's inequality and H\"older continuity of weak solutions. Before Aronson's work \cite{aronson1968non}, these properties were proved for the parabolic equation of divergence form \eqref{div_form} by Moser in \cite{moser1964harnack}, \cite{moser1967correetion} and \cite{moser1971pointwise}, for \eqref{toy1} with bounded drifts by Krylov and Safanov \cite{krylov1980certain}, or for \eqref{toy1} with H\"older continuous drifts by Aronson in \cite{aronson1967bounds}, etc. \par

\begin{theorem}[Aronson, \cite{aronson1968non}; Fabes and Stroock \cite{fabes1989new}] \label{thm_aronson}

Let $n \geq 2$. Given that $u \in L^p(\R^n)$ with $p \geq n$, the parabolic equation \eqref{toy1} have the following properties:

\begin{enumerate}[leftmargin=*,label=\textup{(\arabic*)},align=left]
  \item \textup{[P1] \textbf{Gaussian bounds} (cf. \cite{aronson1968non}, page 613)}. There exist $c_1,c_2,C>0$ depending on $u$, such that the fundamental solution $G$ of \eqref{toy1} satisfies
    \begin{equation} \label{Gaussian_UL}
      C^{-1} g_1(x-y,t-s) \leq G(x,t;y,s) \leq Cg_2(x-y,t-s)
    \end{equation}
  for any $x,y \in \R^n$, $t,s \geq 0$ with $t>s$, where $g_i$ is the fundamental solution of the heat equation $\partial_t v - c_i \Delta v = 0$ for $i = 1,2$.

  \item \textup{[P2] \textbf{Harnack inequality} (cf. \cite{aronson1968non}, Theorem H)}. Let $v$ be a non-negative weak solution of the equation \eqref{toy1} in bounded space-time domain $Q \subset \R^n \times [0,\infty)$. Suppose $Q' \subset Q$ is convex and $d:=\dist(Q', \partial Q) > 0$, then we have
  \begin{equation} \label{harnack} \sup_Q v \leq C \inf_Q v, \end{equation}
  where $C>0$ depends on $d$ and $\|u\|_{L^p(\R^n)}$.

  \item \textup{[P3] \textbf{H\"older continuity} (cf. \cite{aronson1968non}, Theorem C)}. Let $v$ be a bounded weak solution of the equation \eqref{toy1} in bounded space-time domain $\Omega \times [t_0,t_1] \subset \R^n \times [0,\infty)$ with $t_1>t_0$, then $v \in C^\alpha(\Omega'\times[t_0+t',t_1])$ for any $\Omega'$ with $\dist(\Omega, \Omega')>0$, any $t' \in (0,t_1-t_0)$ and some $\alpha > 0$. The $C^{\alpha}$-norm and the H\"older exponent depend on $\Omega$, $\Omega'$, $\frac{t'}{t_1-t_0}$, $\|v\|_{L^2_{x,t}(\Omega \times [t_0,t_1])}$ and $\|u\|_{L^p(\R^n)}$.\footnote{The dependence of the constant $C>0$ is important, but this is not explicitly mentioned in \cite{aronson1968non}. For detailed estimates one can check the proofs of Theorem 2 and Theorem 3 in \cite{aronson1967local}.}
\end{enumerate}

\end{theorem}
\par

Osada \cite{osada1987diffusion} proved that these properties also hold for critical drifts $u \in L^\infty_tL^{\infty,-1}_x$, where $L^{\infty,-1}_x$ denotes the distributions which are the first-order derivatives of bounded measuable functions. Fabes and Stroock \cite{fabes1989new} gave another proof of these properties for the equation of divergence form \eqref{div_form} using De Giorgi-Nash-Moser iteration. We remark that one can follow this approach to prove that [P1], [P2] and [P3] also hold for equation \eqref{toy1} with critical drifts $u \in L^\infty_tL^n_x$. Later, Zhang \cite{zhang2004strong} proved these properties if $u$ is pointwise bounded by $C|x|^{-1}$. Friedlander, Vicol \cite{friedlander2011global} and Seregin, Silvestre, {\v{S}}ver{\'a}k, Zlato{\v{s}} \cite{seregin2012divergence} generalized these properties to drifts in the critical space $L^\infty_t \text{BMO}^{-1}_x$.

An important question that one may ask is whether the Laplace operator or the drift $u$ is dominant in terms of the properties [P1], [P2] and [P3], in particular for more general drifts such as supercritical ones. For most subcritical or critical drifts, the results mentioned above tell us the Laplace operator dominates. For supercritical drifts, very little is known. Indeed, the presence of possibly supercritical drifts is exactly the difficulty which prevents us from confirming or excluding singularities in many important fluid models using standard techniques such as De Giorgi-Nash-Moser iteration. The first contribution of this paper is to construct general supercritical drifts such that [P1], [P2] and [P3] fail, namely to prove the following theorem.

\begin{theorem} \label{thm_parabolic}
Let $n \geq 3$. For any $\lambda \in (0,n-1)$, there exists an time-independent divergence-free drift $u \in L^{n-\lambda}(B)$ with the following property, where $T \in (0,\infty)$ and $B$ is the unit ball of $\,\R^n$. The parabolic equation \eqref{toy1} in $B \times [0,T]$ has a bounded weak solution which is not continuous at the origin at the finite time $T>0$.
\end{theorem}
\par
\begin{remark}
This tells us that the regularity assumption on $u$ in \Cref{thm_aronson} is sharp. In \Cref{subsec_NS}, we shall discuss its implications in fluid dynamics. Our result \Cref{thm_parabolic} is also sharp in dimension since it is false when $n=2$. The details are given in \Cref{lowerdim}.
\end{remark}
\begin{remark}
Results analogues to \Cref{thm_parabolic} hold for many more general supercritical spaces. One can see this in \Cref{thm_NS} or by slightly modifying our construction of $u$.
\end{remark}
\begin{remark}
Since the loss of continuity is a local property, it is not necessary to specify initial and boundary conditions. As we shall see, the loss of continuity can be established for various smooth initial conditions and boundary data.
\end{remark}

Following the approach in \cite{fabes1989new}, one can prove [P1] $\rightarrow$ [P2] $\rightarrow$ [P3]. Therefore, a natural corollary of \Cref{thm_parabolic} is that the Gaussian bounds \eqref{Gaussian_UL} and the Harnack inequality [P2] fail for the supercritical drifts that we construct. Consequently, \Cref{thm_parabolic} means that supercritical drifts can significantly change the fundamental structure of the semigroup generated by the parabolic equation \eqref{toy1}, including the behavior of its fundamental solution.

\subsection{The elliptic equation of non-divergence form}

The elliptic counterpart of \eqref{toy1} is given by
\begin{equation} \label{toy1_elliptic}
  - \Delta v + (u \cdot \nabla) v = 0.
\end{equation}
For any $\lambda>0$, if $v$ is a solution of \eqref{toy1_elliptic} with drift $u$, $v_\lambda(x) := \lambda v(\lambda x)$ is also a solution of \eqref{toy1_elliptic} with drift $u_\lambda(x) := \lambda u(\lambda x)$. In dimensions $n \geq 3$, by integrating the estimate \eqref{nash} in time, we can show an a-priori estimate on $\|v\|_{L^\infty}$ for weak solutions $v \in H^1$ of \eqref{toy1_elliptic}. \par

De Giorgi's work \cite{de1957sulla} for \eqref{div_form} was generalized to \eqref{toy1_elliptic} by Stampacchia \cite{stampacchia1965probleme} in 1965 in the case of subcritical or critical drifts $u$.
\begin{theorem}[Stampacchia, \cite{stampacchia1965probleme}] \label{stampacchia}
Let $n \geq 2$. For any $\lambda \geq 0$, any bounded domain $\Omega \subset \R^n$ and $u \in L^{n+\lambda}(\Omega)$, any weak solution of the elliptic equation \eqref{toy1_elliptic} in $\Omega$ is of $\,C^\alpha(\Omega')$ for any $\Omega' \subset \Omega$ and some $\alpha(\Omega,\Omega',u) > 0$. Moreover, Harnack's inequality holds for non-negative weak solutions.
\end{theorem}

To be consistent with the parabolic case, we denote Harnack's inequality and H\"older continuity in the elliptic case by [E2] and [E3]. Independent of Stampacchia's work, similar problems have been studied by Ladyzhenskaya and Uraltseva \cite{ladyzhenskaya1964holder}, \cite{ladyzhenskaya1973linear}, Morrey \cite{morrey1959second}, \cite{morrey2009multiple}, Moser \cite{moser1960new}, \cite{moser1961harnack} and Serrin \cite{serrin1964local}. Later, there have been many efforts trying to prove [E2] and [E3] for more general drifts, such as \cite{cranston1987conditional}, \cite{osada1987diffusion} and \cite{nazarov2009qualitative}. As far as we know, [E2] and [E3] have been only been proved in general in the case of subcritical or critical drifts. More recently, Seregin, Silvestre, {\v{S}}ver{\'a}k and Zlato{\v{s}} \cite{seregin2012divergence} proved, among other things, that $u \in \text{BMO}^{-1}(\R^n)$ suffices to give [E2] and [E3]. They also conjectured on page 510 of \cite{seregin2012divergence} that the divergence-free condition on $u$ is not sufficient to get a $\,C^\alpha$-bound on the weak solutions of \eqref{toy1_elliptic} under the assumptions $\|v\|_{L^\infty}<C$ and $\|u\|_{L^{n-\lambda}}<C$ with $\lambda>0$.

The second contribution of this paper is to prove that the choice of $n+\lambda$ in \Cref{stampacchia} is sharp and thus to confirm the conjecture of Seregin, Silvestre, {\v{S}}ver{\'a}k and Zlato{\v{s}} \cite{seregin2012divergence}. We prove

\begin{theorem} \label{thm_elliptic}
Let $n \geq 3$. For any $\lambda \in (0,n-1)$, there exists divergence-free $u \in L^{n-\lambda}(\R^n)$ satisfying $\divr u = 0$, such that there is a bounded weak solution $v$ of the elliptic equation \eqref{toy1_elliptic} in the unit ball of $\,\R^n$ which is not continuous at the origin.
\end{theorem}

H\"older continuity can easily be derived from Harnack's inequality. Thus \Cref{thm_elliptic} immediately implies that the Harnack inequality for the elliptic equation \eqref{toy1_elliptic} fails for the supercritical drifts constructed in \Cref{thm_elliptic}. \par

Before our result, we are only aware of an example showing the loss of continuity for a supercritical drift $u \in L^1$ in dimension $n=3$ constructed by Seregin, Silvestre, {\v{S}}ver{\'a}k and Zlato{\v{s}} in \cite{seregin2012divergence}.

\subsection{Connections with the axi-symmetric Navier-Stokes equations} \label{subsec_NS}

Similar to what Nash had in mind, many mathematicians tried to prove uniqueness and smoothness of solutions $u$ to the equations governing fluid motions by assuming some a priori regularity of $u$, but in most cases, there is a big gap between the regularity we need to assume and the regularity we can prove with the available techniques.

An important example is the axi-symmetric Navier-Stokes equations in $\R^3 \times [0,\infty)$. In cylindrical coordinate system of $\R^3$, a velocity field $u=(u_r, u_{\theta}, u_z)$ is given by
\[ u = u_r \mathbf{e}_r(r,\theta,z) + u_{\theta} \mathbf{e}_{\theta}(r,\theta,z) 
     + u_z \mathbf{e}_{z}(r,\theta,z). \]
Assume $(u,p)$ is a solution to the Navier-Stokes equations which is axi-symmetric with respect to $z$ axis, then we have
\begin{equation} \label{navierstokes_cylindrical}
  \begin{split}
  \partial_t u_r + u \cdot \nabla u_r - r^{-1}u_{\theta}^2 + \partial_r p &= \Delta u_r - r^{-2}u_r \\
  \partial_t u_{\theta} + u \cdot \nabla u_{\theta} - r^{-1}u_ru_{\theta} &= \Delta u_{\theta} - r^{-2}u_{\theta} \\
  \partial_t u_z + u \cdot \nabla u_z + \partial_z p &= \Delta u_z \\
  \divr u &= 0
  \end{split}.
\end{equation}
\par

The question of global regularity of the Navier-Stokes equations in $\R^3 \times [0,\infty)$ is still open, even in axi-symmetric case. As a special case which still captures the main properties and difficulties of the original equations, most remarkably the natural scaling invariance, the axi-symmetric equation \eqref{navierstokes_cylindrical} is intensively studied (\cite{ladyzhenskaya1968unique}, \cite{ukhovskii1968axially}, \cite{chen2008lower}, \cite{koch2009liouville}, \cite{chen2009lower}, \cite{seregin2009type}, \cite{lei2011liouville}, \cite{wei2016regularity} etc.). An important observation is that the quantity $\chi = ru_{\theta}$ satisfies the following parabolic equation
\begin{equation} \label{axissymmetricequation1}
  \partial_t \chi - \Delta \chi + (u \cdot \nabla) \chi + 2r^{-1} \partial_r \chi = 0.
\end{equation}
The parabolic equation \eqref{axissymmetricequation1} with nonzero divergence-free drift $u$ has been shown to be very useful for excluding singularities of the first type for axi-symmetric Navier-Stokes equations (\cite{chen2008lower}, \cite{koch2009liouville}, \cite{chen2009lower}, \cite{seregin2009type}, etc.), i.e. the solutions $u$ developing singularities at time $T$ with
\begin{equation} \label{type1}
  |u| \leq \frac{C}{|r|} \quad \text{or} \quad |u| \leq \frac{C}{\sqrt{T-t}}.
\end{equation}
The idea is as follows: Given $u \in L_t^{\infty}L_x^2 \bigcap L_t^2H_x^1$, using the Gaussain bound \eqref{nash}, one can prove that $\chi$ is locally bounded. If we assume $\chi$ is locally H\"older continuous, then by \cite{wei2016regularity} $u$ is regular. By condition \eqref{type1}, the velocity field $u$ is subcritical or critical, hence the H\"older continuity of $\chi$ can be proved using De Giorgi-Nash-Moser iteration (\cite{chen2008lower}, \cite{chen2009lower}).
\par

The global regularity of the axi-symmetric Navier-Stokes equations in $\R^3 \times [0,\infty)$ would be solved if one could show the H\"older continuity of $\chi$ for supercritical drifts $u$ in the energy class $L_t^{\infty}L_x^2 \bigcap L_t^2H_x^1$. This motivates the study of the following toy model in $\R^3$,
\begin{equation} \label{axissymmetricequation2}
  \partial_t v - \Delta v + (u \cdot \nabla) v + 2r^{-1} \partial_r v = 0.
\end{equation}
where $u \in L_t^{\infty}L_x^2 \bigcap L_t^2H_x^1$ is a weakly divergence-free vector field and possibly supercritical. Here we do not assume any relation between $u$ and $v$. However, we show that the H\"older continuity of $v$ in \eqref{axissymmetricequation2} fail for certain supercritical drifts in the natural energy space $L_t^{\infty}L_x^2 \bigcap L_t^2H_x^1$.

\begin{theorem} \label{thm_NS}
For any $\lambda \in (0,2)$, any $\alpha \in (0,\frac{\lambda}{3-\lambda})$ and any $q \in \big[ 1, \frac{4+2\alpha}{1+2\alpha} \big)$, there exists time-dependent divergence-free drift $u \in L^{\infty}_tL_x^{3-\lambda} \bigcap L^q_tH^1_x$ satisfying
\[ |u| \leq \min \Big\{
    \frac{C(\lambda)}{|x|^{1+\alpha}}, \frac{C(\lambda)}{(T-t)^{{(1+\alpha)/}{(2+\alpha)}}} \Big\}
\]
with the following property. The parabolic equation \eqref{axissymmetricequation2} in $B \times [0,T]$, where $T \in (0,\infty)$ and $B$ is the unit ball of $\,\R^3$, has a bounded weak solution which is not continuous at the origin at the finite time $T>0$. In particular, $v/r$ which corresponds to $u_\theta$ in the axi-symmetric Navier-Stokes equations \eqref{navierstokes_cylindrical} blows up with rate $r^{-1}$.
\end{theorem}

\Cref{thm_NS} shows that the conditional regularity result of Chen, Strain, Yau and Tsai (Theorem 3.1 in \cite{chen2008lower}) is sharp. This suggests the toy model of the the axi-symmetric Navier-Stokes equations in space dimension $n=3$ can develop a finite-time singularity at a linear level. Although the dynamics in the linearized equation is significantly simpler than the original nonlinear equation, we still hope this can shed some light and lead to further progress on the Navier-Stokes equations.

\subsection{Lower dimension $n=2$} \label{lowerdim}

Our results \Cref{thm_elliptic} and \Cref{thm_parabolic} are only valid in the dimensions $n \geq 3$. Interestingly, they are false in dimension $n=2$. \par

Indeed, for the solution $v$ of the elliptic equation \eqref{toy1_elliptic} in the unit ball $B \subset \R^2$, Seregin, Silvestre, {\v{S}}ver{\'a}k and Zlato{\v{s}} \cite{seregin2012divergence} showed the following a priori continuity estimate for $v$ (cf. Theorem 1.4 in \cite{seregin2012divergence}): for any $r \in (0,1)$,
\[ \sup_{x \in B_r(0)} |v(x)-v(0)| 
    \leq \frac{C(1+\|u\|_{L^1(B)})^{1/2}}{\sqrt{-\log r}} \|v\|_{L^\infty(B)}. \]
\par

For the solution $v$ of the parabolic equation \eqref{toy1} in $\R^2 \times [0,T]$ with $T>0$, Silvestre, Vicol and Zlato{\v{s}} \cite{silvestre2013loss} proved that, for time-independent drift $u$, we have the following a-prior continuity estimate for $v$ (cf. Theorem 1.4 in \cite{silvestre2013loss}), for any $r \in (0,1)$ and $t \in (0,T)$,
\[ \sup_{x \in B_r(0)} |v(x,t)-v(0,t)| 
    \leq \frac{C(1+\|u\|_{L^1_{\text{loc}}})\|v_0\|_{C^2 \cap W^{4,1}}}{\sqrt{-\log r}}
    \Big(1+\frac{1}{t}\Big). \]
\par


\subsection{Organization of the paper}

In \Cref{chap_drift}, we will give our key construction of the supercritical drifts $u \in L^{n-\lambda}(\R^n)$ for arbitrarily small $\lambda>0$. This construction will serve both the parabolic and the elliptic case. In \Cref{chap_para}, we prove the loss of continuity in the parabolic case, \Cref{thm_parabolic}. In \Cref{chap_NS}, we apply the same idea to the toy model \eqref{axissymmetricequation2} of the axi-symmetric Navier-Stokes equations in space dimension $3$ and give the proof of \Cref{thm_NS}. Finally, in \Cref{chap_elli}, we prove \Cref{thm_elliptic} using a stochastic interpretation of the elliptic equation \eqref{toy1_elliptic}.

\subsection{Acknowledgement}

I am indebted to professor Vladimir {\v{S}}ver{\'a}k for the inspiring discussions and professor Michael Struwe for his careful reading and very helpful feedback.


\section{Construction of supercritical drifts} \label{chap_drift}

In this section we construct suitable supercritical drifts. Our construction is inspired by the work of Silvestre, Vicol and Zlato{\v{s}} \cite{silvestre2013loss}, where they proved the loss of continuity for the solutions of the parabolic equations with fractional Laplacian in dimension $n=2$. We design supercritical drifts in a way such that the velocity field in a cone area transports the scalar $v$ to the origin of the cone. Then the loss of continuity takes place at the origin, where the drift $u$ is singular. \par

However, as we discussed in \Cref{lowerdim}, the cases of dimension $n=2$ and dimensions $n \geq 3$ are very different. In dimensions $n \geq 3$, a direct representation of velocity fields by their stream functions is not available. For this reason, we set up a product-type coordinate system in $\R^n$ mixed by Cartesian coordinate in $\R$ and hyperspherical coordinate in $\R^{n-1}$. For the velocity fields with radial symmetry in hyperspherical coordinate system of $\R^{n-1}$, we propose a notion of generalized stream functions. The function spaces in \cite{silvestre2013loss} are H\"older spaces in $\R^2$. We are interested in $L^{n-\lambda}$ spaces in $\R^n$ for $n \geq 3$, which needs a different design of velocity fields and a more refined control on the set of singularities.

Let $(x_1, x_2, \ldots, x_n) \in \R^n$ be the Cartesian coordinate in $\R^n$. Consider the scalar equation \eqref{toy1} in a mixed coordinate system of $\R^n$
$$(r,z,\theta,\varphi_1,\ldots,\varphi_{n-3}) \in [0,\infty) \times \R \times \Big[-\frac{\pi}{2},\frac{3\pi}{2}\Big) \times [0,\pi]^{n-3} $$
with the transformation
\begin{equation} \label{transformation}
	\begin{cases}
	x_1 = z \\
	x_2 = r \cos\theta \\
	x_3 = r \sin\theta \cos\varphi_1 \\
	x_4 = r \sin\theta \sin\varphi_1 \cos\varphi_2 \\
	\quad \ldots \\
	x_{n-1} = r \sin\theta \sin\varphi_1 \ldots \cos\varphi_{n-3} \\
	x_n = r\sin\theta \sin\varphi_1 \ldots \sin\varphi_{n-3}
	\end{cases}.
\end{equation}
We denote the unit vectors in this mixed coordinate system by $(\mathbf{e}_r, \mathbf{e}_z, \mathbf{e}_\theta, \mathbf{e}_{\varphi_1}, \ldots)$. In Cartesian coordinate system, we use the traditional notation $(\mathbf{e}_1, \mathbf{e}_2, \mathbf{e}_3, \mathbf{e}_4, \ldots)$. In this paper, we use $|\cdot|$ to denote standard Euclidean norm. Since the drift we construct only depends on $r$ and $z$, for the simplicity of notation, we also use $|(r,z)|:=\sqrt{r^2+z^2}$.
\par

For the parabolic case, namely throughout \Cref{chap_para} and \Cref{chap_NS}, we consider functions which only depend on $r,\,z$, $\theta$ and $t$. For those functions $f$ which are independent of $\varphi_i, 1 \leq i \leq n-3$, we will only write the first four effctive arguments $f(r,z,\theta,t)$. The gradient of $f$ is given by
\[ \nabla f = \mathbf{e}_r \partial_r f + \frac{\mathbf{e}_\theta}{r} \partial_\theta f + \mathbf{e}_z \partial_z f \]
and the Laplacian is given by
\[
	\Delta f = \frac{1}{r^{n-2}} \partial_r( r^{n-2} \partial_r f ) 
		+ \frac{n-3}{r^2\tan \theta} \partial_\theta f
		+ \frac{1}{r^2} \partial_{\theta\theta} f
		+ \partial_{zz} f.
\]
\par

We also present a method to construct divergence-free vector fields with radial symmetry in hyperspherical part $\R^{n-1}$ from scalar-valued functions.

\begin{lemma} \label{div_free}
Let $n \geq 3$. For any locally integrable function $\Psi:[0,\infty) \times \R \rightarrow \R$ which belongs to $C^2$ in any compact subset of $([0,\infty) \times \R) \backslash \{0\}$ and the velocity field $u := (u_r, u_z, 0, \ldots, 0): \R^n \rightarrow \R^n$ pointwisely defined by
\begin{equation} \label{stream} 
	(u_r, u_z) := \Big( -\frac{\partial_z \Psi}{r^{n-2}}, \frac{\partial_r \Psi}{r^{n-2}} \Big),
\end{equation}
assume the pointwise derivative $\nabla u$ belongs to $L^1(\R^n)$. Then $u \in W^{1,1}_{\text{loc}}(\R^n)$ and $u$ is a weakly divergence-free drift. The scalar function $\Psi$ is called a generalized stream function.
\end{lemma}

\begin{proof}
First, we compute the divergence operator for velocity field $u := (u_r, u_z, 0, \ldots, 0)$ which only depends on $r$ and $z$ in the mixed coordinate system. Let $x:=(y,z) \in \R^{n-1} \times \R$, the Cartesian coordinate of $u$ is given by $$\Big( \frac{y}{|y|}u_r, u_z \Big).$$ Then computing its divergence directly in Cartesian coordinate gives
\[ \divr u = \partial_r u_r + \frac{n-2}{r} u_r + \partial_z u_z. \]
Note that $\divr u = 0$ formally corresponds to
\begin{equation}  \label{div_free_eq2} \partial_r(r^{n-2}u_r) + \partial_z(r^{n-2}u_z) = 0. \end{equation}
\par
Since the set $\{0\}$ has $W^{1,1}(\R^n)$-capacity zero and $\nabla u \in L^1_{\text{loc}}([0,\infty) \times \R)$, $u \in W^{1,1}_{\text{loc}}(D)$. One can easily verify  $\divr u = 0$ by a formal computation.
\end{proof}

Now, we construct divergence-free drifts from their generalized stream functions $\Psi$.

\begin{lemma} \label{velocity}
Fix any $\lambda \in (0,n-2)$ and any $\alpha \in (0,\frac{\lambda}{n-\lambda})$, define the generalized stream function $\Psi|_D: D:=[0,2] \times [-2,2] \rightarrow \R$ as follows. Let
\begin{align} \label{stream_potential}
	\Psi=\frac{1}{2(n-2-\alpha)}\Psi_s,
\end{align}
where $\Psi_s$ is given by
\begin{equation*}
	\begin{split}
    \Psi_s(r,z) = \begin{cases}
    	-r^{n-1}, &0 \leq \frac{3r}{4} \leq -z \\
    	(r+z)^{n-2-\alpha} - (r-z)^{n-2-\alpha}, \quad &-\frac{r}{2} \leq z \leq \frac{r}{2} \\
    	r^{n-1}, &0 \leq \frac{3r}{4} \leq z
    \end{cases}
    \end{split}
\end{equation*}
and
\begin{align*}
	\Psi_s(r,z) =& -r^{n-1} \varrho\Big(\frac{-4z-2r}{r}\Big) \\
		&+ \big[ (r+z)^{n-2-\alpha} - (r-z)^{n-2-\alpha} \big] \varrho\Big(\frac{3r+4z}{r}\Big),
			\,\, \frac{r}{2} \leq -z \leq \frac{3r}{4}, \\
	\Psi_s(r,z) =& r^{n-1} \varrho \Big( \frac{4z-2r}{r} \Big) \\
		&+ \big[ (r+z)^{n-2-\alpha} - (r-z)^{n-2-\alpha} \big] \varrho \Big( \frac{3r-4z}{r} \Big),
			\,\, \frac{r}{2} \leq z \leq \frac{3r}{4}.
\end{align*}
Here, $\varrho: \R \rightarrow \R$ is a non-negative smooth function with $\varrho(s) = 0$ for $s \leq 0$ and $\varrho(s) = 1$ for $s \geq 1$. In the region $\{r \geq 2\} \bigcup \{|z| \geq 2\}$, we smoothly extend $\Psi|_D$ such that $(u_r, u_z)$ obtained from \eqref{stream} is smooth in $\{r \geq 2\} \bigcup \{|z| \geq 2\}$ and of exponential decay at infinity. Then $\Psi$ is smooth away from $x=0$, and $u=(u_r, u_z, 0, \ldots, 0)$ is in $L^{n-\lambda}(\R^n)$ with respect to the standard metric $r^{n-2}drdz$ and weakly divergence-free. Furthermore, $u$ satisfies the bound
\[ |u| \leq \frac{C(\lambda)}{|x|^{1+\alpha}}. \]
\end{lemma}

\begin{proof}
Note that $\Psi$ is continuous and odd in $z$. From \eqref{stream}, we can compute the velocity in the following subregion
\begin{equation} \label{velocity_eq0}
	\begin{split}
    (u_r,&u_z) = \\ &\begin{cases}
    	(0,-\frac{n-1}{2(n-2-\alpha)}), &0 \leq \frac{3r}{4} \leq -z \\
    	\big(-\frac{(r+z)^{n-3-\alpha}}{2r^{n-2}} - \frac{(r-z)^{n-3-\alpha}}{2r^{n-2}},
    			\frac{(r+z)^{n-3-\alpha}}{2r^{n-2}} - \frac{(r-z)^{n-3-\alpha}}{2r^{n-2}} \big), \quad &-\frac{r}{2} \leq z \leq \frac{r}{2} \\
    	(0,\frac{n-1}{2(n-2-\alpha)}), &0 \leq \frac{3r}{4} \leq z
    \end{cases},
    \end{split}
\end{equation}
then direct computations show that $u|_D \in L^{n-\lambda}(D), \, \nabla u|_D \in L^1(D)$ and $u$ satisfies the pointwise bound. Because $u$ is in $C^\infty$ outside $D$ and of exponential decay at infinity, $u \in W^{1,1}(\R^n)$ is weakly divergence-free follows from the fact $\nabla u \in L^1(\R^n)$.
\end{proof}

Next, we truncate $u$ at origin to obtain bounded drifts.

\begin{corollary} \label{velocity_trun}
Let $\Psi,u,\lambda$ and $\alpha$ be as in \Cref{velocity}. For any $\epsilon>0$, there exists $\Psi_\epsilon :[0,\infty) \times \R \rightarrow \R$ with 
\begin{equation} \label{velocity_trun_eq0}
	\Psi_\epsilon = \Psi \text{ in } \big( [0,\infty) \times \R \big) 
		\backslash \{ |(r,z)|<\epsilon, \, -r \leq z \leq r \},
\end{equation}
such that $u_\epsilon \in C^\infty(\R^n)$, $u_\epsilon$ is bounded in $L^{n-\lambda}(\R^n)$ uniformly in $\epsilon$, and
\[ |u_\epsilon| \leq \frac{C(\lambda)}{|x|^{1+\alpha}}, \quad \text{uniformly in } \epsilon, \]
where $u_\epsilon = (u_{\epsilon,r}, u_{\epsilon,z}, 0, \ldots, 0)$ is given by $\Psi_\epsilon$ via \eqref{stream}. Moreover, $\Psi_\epsilon$ is odd in $z$.
\end{corollary}

\begin{proof}
In the region $\{ |(r,z)|<\epsilon, \, -r \leq z \leq r \}$, we smoothly extend $\Psi_\epsilon$ while preserving the symmetry. Because $u$ in \eqref{velocity_eq0} has only one singularity at the origin $(r,z)=(0,0)$, one just needs to truncate this singularity while making $\Psi_\epsilon$ equal to $\pm r^{n-2}$ in $r$ close to $\{r=0\}$.
\end{proof}

\section{The evolution in parabolic case}  \label{chap_para}

In this section, we prove the loss of continuity in the parabolic case. Instead of considering the drift-diffusion equation \eqref{toy1} in $\R^n \times [0,T]$ for some $T>0$, we study its evolution in $D_n \times [0,T]$ with zero boundary condition, where $D_n \subset \R^n$ is the bounded domain defined by
$$D_n:= \big\{(r,z,\theta,\varphi_1,\ldots,\varphi_{n-3}) \in \R^n \,\big|\, (r,z) \in [0,2] \times [-2,2]  \big\}.$$
This is sufficient because of the local nature of the properties [P2] and [P3]. Note that the drift $u$ and all other functions do not depend on the variables $\varphi_i$ for $1 \leq i \leq n-3$. Thus we omit these coordinates for the sake of simplicity from now on. The main result of this section is

\begin{theorem} \label{main1}
Given any $\lambda \in (0,n-2)$ and any $\alpha \in (0,\frac{\lambda}{n-\lambda})$, there exist initial data $v_0 \in C_c^\infty(D_n)$ with $\|v_0\|_{C^2(D_n)} \leq C$, independent of $\varphi_i, 1 \leq i \leq n-3$, and $\kappa>0$ which satisfy the following property. For any $\delta \in (0,\frac{1}{2})$ and any $\epsilon \in (0,\delta]$, for the drift $u_{\epsilon/2}$ given by \Cref{velocity_trun}, the unique classical solution $v_{\epsilon}$ of \eqref{toy1} in the space-time domain $D_n \times [0, \frac{1}{2+\alpha}]$ with initial data $v_0$ and zero boundary condition satisfies
\[ v_\epsilon \Big( \delta,0,0,\frac{1-\delta^{2+\alpha}}{2+\alpha} \Big) \geq \kappa, 
    \quad v_\epsilon \Big( \delta,0,\pi,\frac{1-\delta^{2+\alpha}}{2+\alpha} \Big) \leq -\kappa, \]
and
\begin{align*}
    v_\epsilon(r,z,\theta,t) &\geq 0 \quad \text{ for } |\theta| \leq \frac{\pi}{2},\\
    v_\epsilon(r,z,\theta,t) &\leq 0 \quad \text{ for } \frac{\pi}{2} \leq \theta \leq \frac{3\pi}{2}.
\end{align*}
Here, $C$ is absolute. $\kappa$ is independent of $\epsilon$ and $\delta$, but depends on $\lambda$ and $v_0$.
\end{theorem}

\Cref{thm_parabolic} follows from \Cref{main1}, since we can use the solutions $v_\epsilon$ for the regularized drifts $u_{\epsilon/2}$ to approximate the solution $v$ we need in \Cref{thm_parabolic} with the following convergence argument.

\begin{proof}[Proof of \Cref{thm_parabolic}]
Let $T = \frac{1}{2+\alpha}$. $\{v_\epsilon\}_{\epsilon > 0}$ is uniformly bounded in $L^\infty_tL^2_x \bigcap L^2_tH^1_x$ by standard energy estimates using the fact that $u_{\epsilon/2}$ is divergence-free. Thus a subsequence $v_\epsilon \rightarrow v$ weakly in $L^2_tH^1_x$ and weakly-$*$ in $L^\infty_tL^2_x$, as $\epsilon \rightarrow 0$, and we may choose further subsequences later. Moreover, by the maximum principle, we also know $\{v_\epsilon\}_{\epsilon > 0}$ is uniformly bounded in $D_n \times [0,T]$.\par

Because of the smoothness of $u_{\epsilon/2}$ and the relation \eqref{velocity_trun_eq0}, $u_{\epsilon/2}$ is uniformly bounded in $L^n$ in any compact subset of $(D_n \backslash \{0\}) \times [0,T]$. Moreover, from \Cref{main1}, $v_\epsilon \geq 0$ in the domain $(D_n \bigcap \{|\theta|<\frac{\pi}{2}\}) \times [0,T]$. By \Cref{thm_aronson}, moreover $\{v_\epsilon\}_{\epsilon > 0}$ is uniformly H\"older continuous in any compact subset of $(D_n \bigcap \{|\theta|<\frac{\pi}{2}\}) \times [0,T]$. Indeed, we may use Harnack's inequality \eqref{harnack} with a constant $C$ that is uniform in $\epsilon$. This is because $\|u_{\epsilon/2}\|_{L^n(\Omega)} < C(\Omega)$ for any compact subset $\Omega$ of $D_n \backslash \{0\}$\footnote{Because $u$ defined by \eqref{stream_potential} merely belongs to $L^{n-}(D_n)$, $\{u_{\epsilon/2}\}_{\epsilon>0}$ cannot be bounded uniformly in $L^n(D_n)$. Therefore, the constant $C(\Omega)$ cannot be made uniform in different $\Omega$. This is the reason why we must get rid of the point $x=0$ and use a diagonal argument in this proof.}. Then we can prove uniform H\"older continuity. Note that proving H\"older continuity from Harnack's inequality is a local argument, where we do not need any information on the boundary close to $r=0$. In contrast, if we wanted to use Schauder estimate, we would need some control on the boundary close to $r=0$. \par

Clearly, $u_{\epsilon/2}$ converges to $u$ in $L^{n-\lambda}$ as $\epsilon \rightarrow 0$, where $u$ is given by \Cref{velocity}. With all the convergence information above, we can deduce that $v \in L^\infty_tL^2_x \bigcap L^2_tH^1_x$ solves \eqref{toy1} with drift $u$ in the sense of distributions, thus $v$ is a weak solution. \par

From \Cref{main1}, we know that for any $\delta \in (0,\frac{1}{2})$ and any $\epsilon \in (0,\delta]$,
\[ v_\epsilon\Big( \delta,0,0,\frac{1-\delta^{2+\alpha}}{2+\alpha} \Big) \geq \kappa, 
    \quad v_\epsilon\Big( \delta,0,\pi,\frac{1-\delta^{2+\alpha}}{2+\alpha} \Big) \leq -\kappa. \]
Fix any sequence $\{\delta_k\}_{k \in \N} \subset \R^+$ with $\lim_{k \rightarrow \infty} \delta_k = 0$. By the uniform H\"older continuity in any compact subset of $(D_n \bigcap \{|\theta|<\frac{\pi}{2}\}) \times [0,T]$ and a diagonal argument, we can extract a subsequence from $\{v_\epsilon\}_{\epsilon > 0}$ converging to $v$ such that
\[  v\Big( \delta_k,0,0,\frac{1-\delta_k^{2+\alpha}}{2+\alpha} \Big) \geq \kappa,
    \quad v\Big(\delta_k,0,\pi,\frac{1-\delta_k^{2+\alpha}}{2+\alpha} \Big) \leq -\kappa, \text{ for any }k, \]
holds in the limit when $\epsilon \rightarrow 0$. Here, the second inequality can be proved using exactly the same argument as the first inequality. Letting $k \rightarrow \infty$ shows that $v$ is not continuous at the point $(r,z,\theta,t) = (0, 0, 0, \frac{1}{2+\alpha})$. Note that the loss of continuity is a local behavior, so we can state this result for general space domains, for example, the unit ball in the statement of \Cref{thm_parabolic}.

\end{proof}

To prepare for the proof of \Cref{main1}, we need two lemmas.

\begin{lemma} \label{subsolution}
Given the drift $u \in L^{n-\lambda}(\R^{n})$ from \eqref{stream} and \eqref{stream_potential}, let $\eta \in C^\infty(\R)$ be an even function which is supported on $[-1,1]$. Assume furthermore $\eta(0)=1$ and that $\eta$ is strictly decreasing on $[0,1]$. Then there exists $r_0 \in (0,\frac{1}{2})$ with the following property. Define
\begin{align} \label{subsolution_eq0}
\begin{split}
    \phi(r,z) &:= \eta \Big( \frac{|(r-1,z)|}{r_0} \Big), \\
    g(r,z,\theta,\varphi_1,\ldots,\varphi_{n-3},t) &:= \phi \Big( \frac{r}{h(t)}, \frac{z}{h(t)} \Big) \eta\Big( \frac{8\theta}{\pi} \Big),
\end{split}
\end{align}
where $h(t) := \big(1-(2+\alpha)t \big)^{\frac{1}{2+\alpha}}$, then $g$ satisfies
\begin{equation}
    \partial_t g + u \cdot \nabla g \leq 0, \quad \text{in } D_n \times \Big[0,\frac{1}{2+\alpha}\Big].
\end{equation}
\end{lemma}

\begin{lemma} \label{diffusion_control}
There exist a function $\eta \in C^\infty(\R)$ satisfying all requirements in \Cref{subsolution} and a constant $c_0 > 0$ depending on $\eta$, such that the function $g$ defined in \eqref{subsolution_eq0} satisfies
\[ -\Delta g \leq c_0 [h(t)]^{-2} g. \]
\end{lemma}

Now we prove the first main result.

\begin{proof}[Proof of \Cref{main1}]
For the smooth drift $u_{\epsilon/2} \in C^\infty(\R^n)$ given in \Cref{velocity_trun}, standard theory yields existence and uniqueness of a classical solution $v_\epsilon$ of \eqref{toy1} in the domain $D_n \times [0,\frac{1}{2+\alpha}]$ with zero Dirichlet boundary condition and initial data
\[ v_{\epsilon}(r,z,\theta,0) := g(r,z,\theta,0) - g(r,z,\pi-\theta,0).\]
Super-solutions and sub-solutions satisfy the standard maximum principle. Since this classical solution $v_\epsilon$ is unique, it must preserve all space symmetry of the initial data at any time $t$. At time $t=0$, the initial data has the following symmetry,
\[ v_{\epsilon}(r,z,\theta,0) = -v_{\epsilon}(r,z,\pi-\theta,0), \]
then at all time $t$, we know
\[ v_{\epsilon}(r,z,\theta,t) = -v_{\epsilon}(r,z,\pi-\theta,t). \]
and thus we have
\begin{align} \label{main1_eq1}  v_{\epsilon}\Big(r,z,\frac{\pi}2,t\Big) = v_{\epsilon}\Big(r,z,-\frac{\pi}2,t\Big) = 0.  \end{align}
\par

Define
\[ \ubar{v}(r,z,\theta,t) = \exp\Big( -c_0 \int_{0}^{t} [h(s)]^{-2} ds \Big) \big(g(r,z,\theta,t) - g(r,z,\pi-\theta,t) \big), \]
where $h, g$ and $c_0$ are defined in \Cref{subsolution} and \Cref{diffusion_control}. Because $u = u_{\epsilon/2}$ in $\{|(r,z)|\geq \frac{\epsilon}{2}\}$ and $\ubar{v}(r,z,\theta,t)=0$ for any $\{|(r,z)|\leq \frac{\epsilon}{2}\}$ and any $t \leq h^{-1}(\epsilon)$, we have that, in the region $\{ |\theta| \leq \frac{\pi}{2}, \, t \leq h^{-1}(\epsilon) \}$,
\begin{align*}
    \partial_t \ubar{v} - \Delta \ubar{v} + (u_{\epsilon/2} \cdot \nabla) \ubar{v} &\leq \exp\Big( -c_0 \int_{0}^{t} [h(s)]^{-2} ds \Big) (\partial_t g + u \cdot \nabla g) \leq 0.
\end{align*}
Here, we use \Cref{subsolution}, \Cref{diffusion_control} and the fact that $g(r,z,\pi-\theta,t) = 0$ in the region $\{ |\theta| \leq \frac{\pi}{2}, \, t \leq h^{-1}(\epsilon) \}$. Therefore, $\ubar{v}$ is a sub-solution of \eqref{toy1} in $D':=D_n \bigcap \{ |\theta| \leq \frac{\pi}{2} \}$ until the time $t=h^{-1}(\epsilon)$.
Clearly, by \eqref{main1_eq1} and the definition of $g$, $\ubar{v}|_{\partial D'} = 0$ and $v_{\epsilon}|_{\partial D'} \geq 0$, and then $v_{\epsilon} - \ubar{v} \geq 0$ on $\partial D'$. Since $v_\epsilon - \ubar{v}$ is a super-solution of \eqref{toy1} in $D'$ until the time $t=h^{-1}(\epsilon)$, maximum principle yields that, for any $\delta \in (0,\frac{1}{2})$ and any $\epsilon \in (0,\delta]$, we have
\begin{equation} \label{main1_eq2}
    \begin{split}
        v_{\epsilon}(\delta,0,0,h^{-1}(\delta)) &\geq \ubar{v}(\delta,0,0,h^{-1}(\delta)) 
            = \exp\Big( -c_0 \int_{0}^{h^{-1}(\delta)} [h(s)]^{-2} ds \Big) \\
            &\geq \exp\Big( -c_0 \int_{0}^{\frac{1}{2+\alpha}} \big(1-(2+\alpha)s \big)^{\frac{-2}{2+\alpha}} ds \Big) := \kappa.
    \end{split}
\end{equation}
Here, we use the fact $\frac{-2}{2+\alpha} > -1$. The bound of $v_{\epsilon}(\delta,0,\pi,h^{-1}(\delta))$ can be proved by the same argument in the symmetric space-time domain. $\kappa>0$ is independent of $\epsilon$ and $\delta$.
\end{proof}

Next, we come back to \Cref{subsolution} and \Cref{diffusion_control}.

\begin{proof}[Proof of \Cref{subsolution}]
Note that $h$ is the solution of the solution of the following ODE
\begin{equation} \label{subsolution_eq2}
    h' = -\frac{1}{h^{1+\alpha}}, \quad h(0) = 1.
\end{equation}
Define a vector field $w = w_r \mathbf{e}_r + w_z \mathbf{e}_z$ with
\[ (w_r, w_z) = \Big( -\frac{r}{[h(t)]^{2+\alpha}}, -\frac{z}{[h(t)]^{2+\alpha}} \Big). \]
The function $g$ solves $\partial_t g + w \cdot \nabla g = 0$. This is immediate from
\begin{align*}
    \partial_t g &= \nabla \phi\Big( \frac{r}{h(t)}, \frac{z}{h(t)} \Big) 
        \cdot \Big( -\frac{rh'(t)}{[h(t)]^2}, -\frac{zh'(t)}{[h(t)]^2} \Big) \eta\Big( \frac{8\theta}{\pi} \Big), \\
    w \cdot \nabla g &= \frac{1}{h(t)} \nabla \phi\Big( \frac{r}{h(t)}, \frac{z}{h(t)} \Big) 
        \cdot \Big( -\frac{r}{[h(t)]^{2+\alpha}}, -\frac{z}{[h(t)]^{2+\alpha}} \Big) \eta\Big( \frac{8\theta}{\pi} \Big).
\end{align*}
To check the inequality $\partial_t g + u \cdot \nabla g \leq 0$. Equivalently, we need to verify
\[
    (w-u) \cdot \nabla g \geq 0.
\]
Note that $\eta \geq 0$. By the definitions of $w$ and $g$, it suffices to show
\begin{align}
    \Big( \frac{(r,z)}{[h(t)]^{2+\alpha}} + (u_r,u_z) \Big) \cdot (r-h(t),z) \geq 0
\end{align}
in the region $\{(r,z,t) \,|\, |(r-h(t),z)| \leq h(t)r_0\}$ where the support of $g$ lies. Here, we use $\eta' \leq 0$. \par

Now we can impose our first condition $r_0 < \frac{1}{4}$ to ensure that the space support of $g(\cdot,\cdot,\theta,t)$ stays in the region $\{(r,z) \,|\, -\frac{r}{2} \leq z \leq \frac{r}{2}\}$ for any $t \in [0,\frac{1}{2+\alpha}]$. As we computed in \eqref{velocity_eq0}, the velocity $(u_r,u_z)$ in this region is given by
\[ \Big(-\frac{(r+z)^{n-3-\alpha}}{2r^{n-2}} - \frac{(r-z)^{n-3-\alpha}}{2r^{n-2}},
                \frac{(r+z)^{n-3-\alpha}}{2r^{n-2}} - \frac{(r-z)^{n-3-\alpha}}{2r^{n-2}} \Big). \]
By change of variable, it suffices to prove
\begin{equation} \label{subsolution_eq4}
    \Big[(r,z) + \Big(-\frac{(r+z)^{n-3-\alpha}}{2r^{n-2}} - \frac{(r-z)^{n-3-\alpha}}{2r^{n-2}},
                \frac{(r+z)^{n-3-\alpha}}{2r^{n-2}} - \frac{(r-z)^{n-3-\alpha}}{2r^{n-2}} \Big) \Big] 
                \cdot (r-1,z) \geq 0
\end{equation}
in the region $\{(r,z) \,|\, |(r-1,z)| \leq r_0\}$. This is equivalent to proving that
\[ f(r,z) := -\frac{r-z-1}{2(r+z)^{\alpha+3-n}} - \frac{r+z-1}{2(r-z)^{\alpha+3-n}} + r^{n-1}(r-1) + z^2r^{n-2} \geq 0 \]
in the region $\{(r,z) \,|\, |(r-1,z)| \leq r_0\}$. We can check by elementary computation that $f$ achieves a local minimum at $(1,0)$ with
\begin{align*}
    \partial_{rz} f(1,0) &= 0, \\
    \partial_{rr} f(1,0) &= 4+2\alpha > 0, \\ \partial_{zz} f(1,0) &= 2(n-2-\alpha) > 0.
\end{align*}
Therefore, we can choose $r_0>0$ small enough to ensure \eqref{subsolution_eq4}.

\end{proof}

\begin{proof}[Proof of \Cref{diffusion_control}]
By scaling, it suffices to prove the following inequality 
\begin{equation} \label{diffusion_control_eq1}
    -\Delta \bar{g} \leq c_0 \bar{g}
\end{equation}
for the function
\[ \bar{g}(r,z,\theta) := \eta \Big( \frac{|(r-1,z)|}{r_0} \Big) \eta\Big( \frac{8\theta}{\pi} \Big). \]
\par

Now fix
\begin{align*}
    \eta(s) := \begin{cases}
        0, & |s| \geq 1 \\
        \exp\big(\frac{|s|}{|s|-1}\big), \quad &\frac{1}{2} \leq |s| < 1 \\
        \text{any smooth even monotone extension,} & \frac{1}{4} \leq |s| \leq \frac{1}{2} \\
        1-s^2, & |s| \leq \frac{1}{4}
    \end{cases},
\end{align*}
Let $\rho := \frac{|(r-1,z)|}{r_0}$, then we compute
\begin{align} \label{diffusion_control_eq2}
\begin{split}
\Delta \bar{g}
    =& \Big( \partial_{rr} + \frac{n-2}{r}\partial_r + \frac{1}{r^2}\partial_{\theta \theta} + \frac{n-3}{r^2\tan\theta} \partial_\theta + \partial_{zz} \Big) \bar{g} \\
    =& \frac{1}{r_0^2} \eta'' (\rho) \eta\Big( \frac{8\theta}{\pi} \Big) 
        + \frac{n-1}{r_0^2 \rho} \eta' (\rho) \eta\Big( \frac{8\theta}{\pi} \Big)
        - \frac{n-2}{r_0^2 r\rho} \eta' (\rho) \eta\Big( \frac{8\theta}{\pi} \Big) \\
        &+ \frac{64}{\pi^2r^2} \eta (\rho) \eta''\Big( \frac{8\theta}{\pi} \Big)
        + \frac{8(n-3)}{\pi r^2 \tan\theta} \eta(\rho) \eta'\Big( \frac{8\theta}{\pi} \Big) \\
    \geq& \frac{1}{r_0^2} \eta\Big( \frac{8\theta}{\pi} \Big) \Big( \eta''(\rho) + \frac{n-1}{\rho} \eta' (\rho) \Big)
        + \frac{64}{\pi^2r^2} \eta (\rho) \Big( \eta''\Big( \frac{8\theta}{\pi} \Big) 
        + \frac{\pi(n-3)}{8\tan\theta} \eta'\Big( \frac{8\theta}{\pi} \Big) \Big) .
\end{split}
\end{align}
Here we eliminate the third term in the second line of \eqref{diffusion_control_eq2}, because $-\eta',\eta \geq 0$ on $[0,\infty)$. \par

By the definition of $\eta$, we have that for $\frac{1}{2} \leq s \leq 1$, the following two terms are bounded from below for $c_1<0$ small enough,
\begin{align*}
    \Big( \eta''(s) + \frac{\pi(n-3)}{8\tan(\frac{\pi s}{8})} \eta'(s) \Big) \frac{1}{\eta(s)} 
        &= \frac{1}{(s-1)^4} + \frac{2}{(s-1)^3} 
        - \frac{\pi(n-3)}{8\tan(\frac{\pi s}{8})(s-1)^2} \geq c_1, \\
    \Big(\eta''(s) + \frac{n-1}{s} \eta'(s) \Big) \frac{1}{\eta(s)} &= \frac{1}{(s-1)^4} + \frac{2}{(s-1)^3} - \frac{n-1}{(s-1)^2s} \geq c_1.
\end{align*}
This is indeed true since the first terms on the right hand side dominate when $s$ is close to 1. \par

The same lower bound is also true when $|s| \leq \frac{1}{4}$, i.e.
\begin{align*}
    \Big( \eta''(s) + \frac{\pi(n-3)}{8\tan(\frac{\pi s}{8})} \eta'(s) \Big) \frac{1}{\eta(s)} &= \frac{-2}{1-s^2} \Big( 1 + \frac{2\pi (n-3)s}{8\tan(\frac{\pi s}{8})} \Big) \geq c_1, \\
    \Big(\eta''(s) + \frac{n-1}{s} \eta'(s) \Big) \frac{1}{\eta(s)} &= \frac{-2n}{1-s^2} \geq c_1.
\end{align*}

In the region $\frac{1}{4} \leq |s| \leq \frac{1}{2}$, the same lower bound is also true, because $\eta(s)$ is larger than a positive constant and $|\eta''(s)|, |\frac{1}{s} \eta'(s)|$ are bounded. \par

Combining all these estimates with \eqref{diffusion_control_eq2}, we have
\[
    \Delta \bar{g} \geq \frac{c_1}{r_0^2} \eta\Big( \frac{8\theta}{\pi} \Big) \eta(\rho)
            + \frac{64c_1}{\pi^2r^2} \eta\Big( \frac{8\theta}{\pi} \Big) \eta (\rho)
        = \Big( \frac{c_1}{r_0^2} + \frac{64c_1}{\pi^2r^2} \Big) \bar{g}.
\]
The proof is complete.
\end{proof}

\section{A parabolic toy model for the Navier-Stokes equations}  \label{chap_NS}

In this section, we give the proof of \Cref{thm_NS}. Its difference from \Cref{thm_parabolic} is that we need a drift $u \in L^q_tH^1_x(B \times [0,T])$ for $1 \leq q < \frac{2(2+\alpha)}{2\alpha+1}$. To achieve this goal, we will do some space-time truncation to get a time-dependent drift $u$.

For \eqref{toy1} in dimension $n=3$, we have constructed the drift $\bar{u} := (\bar{u}_r, \bar{u}_z, 0) \in L^{3-\lambda}(\R^3)$ with
\[
 (\bar{u}_r, \bar{u}_z) = \begin{cases}
        (0,-\frac{1}{1-\alpha}), &0 \leq \frac{3r}{4} \leq -z \\
        \text{interpolation region}, &\frac{r}{2} \leq -z \leq \frac{3r}{4} \\
        \big(-\frac{(r+z)^{-\alpha}}{2r} - \frac{(r-z)^{-\alpha}}{2r},
                \frac{(r+z)^{-\alpha}}{2r} - \frac{(r-z)^{-\alpha}}{2r} \big), \quad &-\frac{r}{2} \leq z \leq \frac{r}{2} \\
        \text{interpolation region}, &\frac{r}{2} \leq z \leq \frac{3r}{4} \\
        (0,\frac{1}{1-\alpha}), &0 \leq \frac{3r}{4} \leq z
    \end{cases}
\]
from its generalized function $\bar{\Psi}=\Psi$, defined in \Cref{velocity} with $n=3$. To compensate the term $2r^{-1}\partial_r v$ in the toy model of the axis-symmetric Navier-Stokes equations \eqref{axissymmetricequation2}, define another generalized stream function $\Phi_0: D \rightarrow \R$ by
\begin{align*}
    \Phi_0(r,z) = \begin{cases}
        -r^2, & 0 \leq \frac{3r}{4} \leq -z \\
        -r^2 \varrho\big(\frac{-4z-2r}{r}\big) + 2z \varrho\big(\frac{3r+4z}{r}\big), & \frac{r}{2} \leq -z \leq \frac{3r}{4} \\
        2z, \quad &-\frac{r}{2} \leq z \leq \frac{r}{2} \\
        r^2 \varrho\big(\frac{4z-2r}{r}\big) + 2z \varrho\big(\frac{3r-4z}{r}\big), & \frac{r}{2} \leq z \leq \frac{3r}{4} \\
        r^2, & 0 \leq \frac{3r}{4} \leq z
    \end{cases}
\end{align*}
and extend it outside $D$ smoothly with exponential decay. \par

Now we do space-time truncation such that the drift belongs to the natural energy space of the Navier-Stokes equations. Define a cutoff function $$\bar{\varrho}(r,t) := \varrho\Big( \frac{8r}{h(t)}-1 \Big),$$ where $\varrho$ and $h$ are defined in \Cref{velocity} and \Cref{subsolution} respectively. Define
\begin{align} \label{ns_velocity}
    \Phi := (\Phi_{0} + \bar{\Psi}) \bar{\varrho}
\end{align}
and $\tilde{u}:=(\tilde{u}_{r},\tilde{u}_{z},0)$ via $(\tilde{u}_{r}, \tilde{u}_{z}) = \big( -\frac{\partial_z \Phi}{r^{n-2}}, \frac{\partial_r \Phi}{r^{n-2}} \big)$. Using the same argument in the proof of \Cref{velocity}, one can check $\tilde{u} \in L^{\infty}_tL_x^{3-\lambda}(\R^3)$. And direct computations verify $\tilde{u} \in L^q_tH^1_x(\R^3)$ with $1 \leq q < \frac{2(2+\alpha)}{2\alpha+1}$ for $\alpha>0$ small enough.
\par

As in \Cref{velocity_trun}, we do space truncation to get smooth drifts $\{\Phi_{0,\epsilon}\}_{\epsilon>0}$. Then 
\begin{align} \label{ns_velocity_trun}
    \Phi_{\epsilon} := (\Phi_{0,\epsilon} + \bar{\Psi}_\epsilon) \bar{\varrho}
\end{align}
gives a smooth time-dependent drift $\tilde{u}_{\epsilon}:=(\tilde{u}_{r,\epsilon},\tilde{u}_{z,\epsilon},0)$ via $(\tilde{u}_{r,\epsilon}, \tilde{u}_{z,\epsilon}) = \big( -\frac{\partial_z \Phi_{\epsilon}}{r^{n-2}}, \frac{\partial_r \Phi_{\epsilon}}{r^{n-2}} \big)$. Here, $\bar{\Psi}_\epsilon=\Psi_\epsilon$ is defined in \Cref{velocity_trun} with $n=3$. The velocity field $\tilde{u}_\epsilon$ satisfies
\begin{align} \label{ns_fact1}
    \tilde{u}_{\epsilon} = \bar{u} + \Big( -\frac{2}{r},0 \Big)
        \quad \text{for } -\frac{r}{2} \leq z \leq \frac{r}{2}, r \geq \frac{h(t)}{4}.
\end{align}
\par

In this setting, the following theorem is an analogue of \Cref{main1}.
\begin{theorem} \label{main1_ns}
Given any $\lambda \in (0,1)$ and any $\alpha \in (0,\frac{\lambda}{3-\lambda})$, there exist initial data $v_0 \in C_c^\infty(D_3)$ with $\|v_0\|_{C^2(D_3)} \leq C$, and $\kappa>0$ which satisfy the following property. For any $\delta \in (0,\frac{1}{2})$ and any $\epsilon \in (0,\delta]$, for the drift $\tilde{u}_{\epsilon/2}$ given by \eqref{ns_velocity_trun}, the unique classical solution $v_{\epsilon}$ of \eqref{axissymmetricequation2} in the space-time domain $D_3 \times [0, \frac{1}{2+\alpha}]$ with initial data $v_0$ and zero boundary condition satisfies
\[ v_\epsilon \Big( \delta,0,0,\frac{1-\delta^{2+\alpha}}{2+\alpha} \Big) \geq \kappa, 
    \quad v_\epsilon \Big( \delta,0,\pi,\frac{1-\delta^{2+\alpha}}{2+\alpha} \Big) \leq -\kappa, \]
and
\begin{align*}
    v_\epsilon(r,z,\theta,t) &\geq 0 \quad \text{ for } |\theta| \leq \frac{\pi}{2},\\
    v_\epsilon(r,z,\theta,t) &\leq 0 \quad \text{ for } \frac{\pi}{2} \leq \theta \leq \frac{3\pi}{2}.
\end{align*}
Here, $C$ is absolute. $\kappa$ is independent of $\epsilon$ and $\delta$, but depends on $\lambda$ and $v_0$.
\end{theorem}

The proofs of \Cref{thm_NS} and \Cref{main1_ns} are essentially the same argument in the previous section with the following two facts. The first fact is that, from \eqref{ns_fact1}, in the region $\big\{ (r,z,\theta,t) \in D_3 \times [0,\frac{1}{2+\alpha}] \, | \, -\frac{r}{2} \leq z \leq \frac{r}{2}, r \geq \frac{h(t)}{4} \big\}$, \eqref{toy1} with the drift $u$ given by \Cref{velocity} is identical to \eqref{axissymmetricequation2} with the drift $\tilde{u}$ given by \eqref{ns_velocity}. The second fact is that the effective dynamics we rely on to prove \Cref{main1} and \Cref{thm_parabolic} exactly lies in this region.

\section{The elliptic case} \label{chap_elli}

We consider the elliptic equation \eqref{toy1_elliptic} in the unit ball $B_1(0,\R^n) \subset \R^n$. We use the stochastic formulation of \eqref{toy1_elliptic} to prove the loss of continuity. Similar idea has been used by Seregin, Silvestre, {\v{S}}ver{\'a}k and Zlato{\v{s}} \cite{seregin2012divergence} to prove the loss of continuity for a supercritical drift $u \in L^1$ in dimension $n=3$. For some technical convenience, we work with Cartesian coordinate system in this section. The main result of this section is as follows.
\par

\begin{theorem} \label{prop_elliptic_c}
There exist absolute constants $C, \alpha_0 > 0$ and another constant $\kappa>0$ with the following property. For any $\lambda \in (0,n-1)$, any $\alpha \in (0,\alpha_0)$, any $\delta \in (0,1)$ and any $\epsilon \in (0,\frac{\delta}{2})$, consider the elliptic equation \eqref{toy1_elliptic} with the drift $\,Cu_\epsilon$ in the unit ball with the following smooth boundary condition $\gamma: \partial B_1(0,\R^n) \rightarrow \R$,
\[
    \gamma(x) = \begin{cases}
    1, \quad x_2>\frac{1}{2} \\
    \text{smooth extension which is odd and monotone in } x_2, \quad x_2 \in (-\frac{1}{2}, \frac{1}{2}) \\
    -1, \quad x_2<-\frac{1}{2}
    \end{cases},
\]
then the unique classical solution $v_\epsilon$ of this elliptic equation satisfies
\[ v_\epsilon(0,\delta,0,\ldots,0) \geq \kappa \]
and
\[ v_\epsilon(x) \geq 0, \quad \text{for } x \text{ with } x_2 \geq 0. \]
Here, $u_\epsilon$ is given by \Cref{velocity_trun} and $\kappa>0$ is independent of $\epsilon$ and $\delta$.
\end{theorem}

To prove \Cref{thm_elliptic}, one just needs to follow the proof of \Cref{thm_parabolic} line by line, with the help of \Cref{prop_elliptic_c}. To prove \Cref{prop_elliptic_c}, we need the following technical lemma.

\begin{lemma} \label{travel}
There exist absolute constants $\nu_0, \alpha_0 \in (0,1)$ with the following property. For any $\mu > 0$, $l \in (0,1)$, $\nu \in (0,\nu_0]$, $C>0$ and $a>0$, define
\[ \Sigma_{\nu,\mu} := \Big\{ (x_1,x_2,\ldots,x_n) \in \R^n \, \Big| \, 
    x_2 \in \Big[ \frac{\mu}{2}, 2\mu \Big], \max\big\{ |x_1|, |x_3|, \ldots, |x_n| \big\} \leq \nu x_2 \Big\} \]
and
\[ 
    \begin{split}
        \Gamma_{\nu,\mu} &:= \big\{ (x_1,x_2,\ldots,x_n) \in \Sigma_{\nu,\mu} \, \big| \, 
        x_2 = 2\mu \big\}, \\
        S_{\nu,\mu,l} &:= \big\{ (x_1,x_2,\ldots,x_n) \in \R^n \, \big| \, 
        x_2 = \mu, \max\big\{ |x_1|, |x_3|, \ldots, |x_n| \big\} \leq l \nu \mu \big\}.
    \end{split}
\]
Suppose we have two continuous functions $\rho: [0,a] \rightarrow \Sigma_{\nu,\mu},\, b: [0,a] \rightarrow \R^n$ and for any $t \in [0,a]$,
\begin{align} 
    \rho(t) &= \rho(0) - C\int_{0}^{t} u ( \rho(s) ) ds + b(t) - b(0), \label{travel_eq2}\\
    \sup_t |b(t)-b(0)| &\leq \frac{(1-l)\nu\mu}{8}\label{travel_eq5}.
\end{align}
Moreover, suppose $\rho$ satisfies
\begin{align*}
    \rho(0) &\in S_{\nu,\mu,l},\\
    \rho(a) &\in \partial \Sigma_{\nu,\mu},
\end{align*}
then $\rho(a) \in \Gamma_{\nu,\mu}$. Here, $u$ is given by \Cref{velocity}.
\end{lemma}
\par

Now we can prove \Cref{prop_elliptic_c} from the viewpoint of Brownian motion with drifts. 

\begin{proof}[Proof of \Cref{prop_elliptic_c}]
Consider the elliptic equations with regularized drifts $u_\epsilon$
\begin{equation} \label{toy1_elliptic_reg}
  - \Delta v_\epsilon + C(u_\epsilon \cdot \nabla) v_\epsilon = 0
\end{equation}
with the boundary condition $v_\epsilon|_{\partial B_1(0,\R^n)} = \gamma$. Let $(\Theta, \mathcal{F}, \mathbb{P})$ be a probability space and $\{B_t\}_{t \geq 0}$ be a $n$-dimensional Brownian motion. For any $x \in B_1(0,\R^n)$, let $\{ X_t^{x,\epsilon} \}_{t \geq 0}$ be the stochastic process given by the stochastic differential equation
\begin{align} \label{sde_eps}
  d X_t^{x,\epsilon} &= - Cu_\epsilon (X_t^{x,\epsilon}) dt + dB_t, \\
  X_0^{x,\epsilon} &= x.
\end{align}
The solution of the stochastic differential equation is given by
\[ X_t^{x,\epsilon} = X_0^{x,\epsilon} - \int_{0}^{t} Cu_\epsilon (X_s^{x,\epsilon}) ds + B_t. \]
\par

Define two stopping time $\tau$ and $\vartheta$ by
\[
    \begin{split}
    \tau &:= \inf \{t \geq 0 \,|\, X_t^{x,\epsilon} \in \partial B_1(0,\R^n) \}, \\
    \vartheta &:= \inf \{t \geq 0 \,|\, X_t^{x,\epsilon} \in \partial B_1(0,\R^n), X_t^{x,\epsilon} \cdot \mathbf{e}_2 \leq 0 \}.
    \end{split}
\]
Note that $\mathbb{P}(\tau < \infty) = 1$. By Theorem 9.2.13 of \cite{oksendal2003stochastic}, we know the solution $v_\epsilon$ is given by
\begin{equation} \label{thm2_eq8}
    v_\epsilon(x) = \mathbb{E}[\gamma(X_{\tau}^{x,\epsilon})]
    = \int_{A} \gamma(X_{\tau}^{x,\epsilon}) d\mathbb{P} + \int_{A^c} \gamma(X_{\tau}^{x,\epsilon}) d\mathbb{P},
\end{equation}
where
\[
    \begin{split}
    A := \{ \omega \in \Theta \, | \, X_t^{x,\epsilon}(\omega) \cdot \mathbf{e}_2 > 0 \text{ for any } t \in [0,\tau(\omega)] \}.
    \end{split}
\]
If $x_2>0$, let $\tilde{x} := X_{\vartheta}^{x,\epsilon}$. For any $\omega \in A^c$, $\vartheta(\omega) < \infty$ and $\tilde{x}_2 = 0$. We can deduce
\begin{equation} \label{thm2_eq9}
    \begin{split}
    \int_{A^c} \gamma(X_{\tau}^{x,\epsilon}) d\mathbb{P} = \int_{A^c} \gamma(X_{\tau-\vartheta}^{\tilde{x},\epsilon}) d\mathbb{P} = 0,
    \end{split}
\end{equation}
because the symmetry of $u_\epsilon$ with respect to $x_2=0$ and the oddness of $\gamma$ in $x_2$. Here we also use that Brownian motion has independent increments.

Fix a small number $\nu \in (0,\nu_0)$, where $\nu_0$ is given in \Cref{travel}. For any $y=(y_1,y_2,\ldots,y_n) \in B_1(0,\R^n)$ with $y_2>0$, we introduce more convenient notations for the sets defined in \Cref{travel}
\[ 
    \begin{split}
    \Sigma_y &:= \Sigma_{\nu y_2^{\alpha/4},y_2}, \\
    \Gamma_y &:= \Gamma_{\nu y_2^{\alpha/4},y_2}, \\
    S_y      &:= S_{\nu y_2^{\alpha/4},y_2,2^{-\alpha/4}}.
    \end{split}
\]
Note that $\Gamma_y = S_{2y}$.
\par
\vspace{2mm}

\noindent \textbf{Claim 1:} There exist absolute constants $C,\, C_{1}>0$ with the following property. For any $y \in B_1(0,\R^n)$ with $y_2 \in [2\epsilon,1]$ and $y \in S_y$, the solution $X_t^{y,\epsilon}$ of \eqref{sde_eps} with drift $Cu_{\epsilon}$ satisfies
\begin{equation}
    \mathbb{P} (X_\sigma^{y,\epsilon} \in \Gamma_y \cup \partial B_1(0,\R^n)) \geq 1- \frac{4n}{\sqrt{2\pi}} \exp (-C_1y_2^{-\alpha/2}) \geq \frac{1}{2},
\end{equation}
where the stopping time $\sigma$ is defined by
\[ \sigma := \inf \{t \geq 0 \,|\, X_t^{y,\epsilon} \in \partial (B_1(0,\R^n) \cap \Sigma_y) \}. \] \par
\vspace{2mm}

Since $\tau$ and $\sigma$ are almost surely finite, we only need to look at those occurences for which $\tau$ and $\sigma$ are finite. With the help of Claim 1, we can estimate the value $v_\epsilon(p)$ for a fixed point $p \in B_1(0,\R^n)$ with $p_2 \geq 2\epsilon$ and $p_i=0$ for $i \neq 2$ from \eqref{thm2_eq8}. It suffices to estimate $\mathbb{E}[\gamma(X_{\tau}^{p,\epsilon})]$. Since the Brownian motion trajectory is almost surely continuous and $u_\epsilon$ is bounded, the function $t \rightarrow X_{t}^{p,\epsilon}$ is almost surely continuous. Define $F(\omega) := \{X_s^{p,\epsilon}(\omega), s \in [0,\tau(\omega)] \}$. Consider the sets $\Sigma_{2^k p}$ for $1 \leq k \leq k_0$, where $k_0 := \lfloor -\log_2(p_2) \rfloor$. Let
\[  D := \big\{ \omega \in \Theta \, \big| \, 
    \forall \, k \text{ with } S_{2^kp} \cap F(\omega) \neq \emptyset, t \rightarrow X_{t}^{p,\epsilon} \text{ exits } \Sigma_{2^k p} \text{ from } \Gamma_{2^k p} \big\}. \]
By Claim 1, we have
\[
    \begin{split}
    \mathbb{P} (D) &\geq \prod_{k=1}^{k_0} \Big( 1- \frac{4n}{\sqrt{2\pi}} \exp \big( -C_1 (p_22^k)^{-\alpha/2} \big) \Big) \\
    &\geq \prod_{k=1-k_0}^{0} \Big( 1- \frac{4n}{\sqrt{2\pi}} \exp \big( -C_1 2^{-k\alpha/2} \big) \Big) \\
    &\geq \prod_{k=0}^{\infty} \Big( 1- \frac{4n}{\sqrt{2\pi}} \exp \big( -C_1 2^{k\alpha/2} \big) \Big) \\
    &\geq \kappa > 0.
    \end{split}
\]
Here, $\kappa > 0$ is independent of $\epsilon>0$ and $p_2 \in [2\epsilon, 1]$.
\par

Since $\Gamma_{2^kp} = S_{2^{k+1}p}$ and $p \in S_p$, for any $\omega \in D$, $X_t^{p,\epsilon}(\omega)$ is a trajectory which exits $B_1(0,\R^n)$ from the part $\{ x \in \partial B_1(0,\R^n) \,|\, x_2 > \frac{1}{2} \}$, i.e. for all $\omega \in D$, we have $X_{\tau(\omega)}^{p,\epsilon}(\omega) \cdot \mathbf{e}_2 > \frac{1}{2}$. From \eqref{thm2_eq8} and \eqref{thm2_eq9},
\[
    \begin{split}
    v_\epsilon(p) = \mathbb{E}[\gamma(X_{\tau}^{p,\epsilon})]
    = \int_{A} \gamma(X_{\tau}^{p,\epsilon}) d\mathbb{P} \geq \int_{D} \gamma(X_{\tau}^{p,\epsilon}) d\mathbb{P} \geq \kappa.
    \end{split}
\]
Here, we have the following picture: For those occurences $\omega \in D$, $\gamma(X_{\tau(\omega)}^{p,\epsilon}(\omega)) = 1$. For $\omega \in A$, $\gamma(X_{\tau(\omega)}^{p,\epsilon}(\omega)) \geq 0$. Both follow from the definition of $\gamma$. Finally, we conclude that $v_\epsilon(p) \geq \kappa$ for any $p \in B_1(0,\R^n)$ with $p_2 \in [2\epsilon, 1]$ and $p_i = 0, i \neq 2$. \par

The fact that $v_\epsilon(x) \geq 0$ for $x \in B_1(0,\R^n)$ with $x_2 \geq 0$ follows from maximum principle and symmetry of $\gamma$.

\begin{proof}[Proof of Claim 1]
For $1$-dimensional Brownian motion $B^1_t$, by reflection principle, we have for any $a>0$ and any $t>0$
\[ \mathbb{P}\Big( \sup_{0 \leq s \leq t} B^1_s \geq a \Big)
    = 2 \mathbb{P}\Big( B^1_t \geq a \Big), \]
then
\[ 
    \begin{split}
    \mathbb{P}\Big( \sup_{0 \leq s \leq t} |B^1_s| \geq a \Big)
    \leq& \mathbb{P}\Big( \sup_{0 \leq s \leq t} B^1_s \geq a \Big)
        + \mathbb{P}\Big( \inf_{0 \leq s \leq t} B^1_s \leq -a \Big) \\
    =& 4\mathbb{P}\Big( B^1_t \geq a \Big) 
    = \frac{4}{\sqrt{2\pi t}} \int_a^{\infty} \exp \Big( -\frac{s^2}{2t} \Big) ds.
    \end{split}
\]
\par

For $n$-dimensional Brownian motion $B_t$ and $(nt)^{-\frac{1}{2}}a \geq 1$,
\[ 
    \begin{split}
    \mathbb{P}\Big( \sup_{0 \leq s \leq t} |B_s| \geq a \Big)
    \leq& \sum_{i=1}^{n} \mathbb{P}\Big( \sup_{0 \leq s \leq t} |B^i_s| \geq \frac{a}{\sqrt{n}} \Big) \\
    \leq& \frac{4n}{\sqrt{2\pi t}} \int_{n^{-1/2}a}^{\infty} \exp \Big( -\frac{s^2}{2t} \Big) ds \\
    =& \frac{4n}{\sqrt{2\pi}} \int_{(nt)^{-1/2}a}^{\infty} \exp \Big( -\frac{\zeta^2}{2} \Big) d\zeta \\
    \leq& \frac{4n}{\sqrt{2\pi}} \exp \Big( -\frac{a^2}{2nt} \Big).
    \end{split}
\]
Hence for any $t>0$ and any $a>0$,
\[
    \mathbb{P}\Big( \sup_{0 \leq s \leq t} |B_s| \leq a \Big) \geq 1 - \frac{4n}{\sqrt{2\pi}} \exp \Big( -\frac{a^2}{2nt} \Big).
\]
\par

Consider the case $\Sigma_y \subset B_1(0,\R^n)$. For those trajectories $B_t(\omega)$ with 
\begin{equation} \label{thm2_eq10}
    \sup_{0 \leq s \leq y_2^{2+\alpha}} |B_s(\omega)| \leq \frac{(1-2^{-\alpha/4})\nu y_2^{\alpha/4+1}}{8},
\end{equation}
we can choose $C>0$ big enough such that $t \rightarrow X_{t}^{y,\epsilon}(\omega)$ must exit $\Sigma_y$ at some time $\sigma(\omega) \in (0,C_0^{-1}y_2^{2+\alpha})$ with $C_0>0$ to be determined, because $u_\epsilon \cdot \mathbf{e}_2 \sim y_2^{-(1+\alpha)}$ in $\Sigma_y$, from the definition of the velocity field $Cu_\epsilon$. Since $y \in S_y$ and $u=u_\epsilon$ in $\Sigma_y$, we can apply \Cref{travel} and we can deduce $X_{\sigma(\omega)}^{y,\epsilon}(\omega) \in \Gamma_y$. This leads to
\[ \mathbb{P}\big( X_\sigma^{y,\epsilon} \in \Gamma_y \big)
    \geq 1- \frac{4n}{\sqrt{2\pi}} \exp \big( -C_1 y_2^{-\alpha/2} \big), \quad C_1 := \frac{C_0(1-2^{-\alpha/4})^2 \nu^2}{32n}. \]
We can choose $C_0>0$ such that $\mathbb{P}\big( X_\sigma^{y,\epsilon} \in \Gamma_y \big) \geq \frac{1}{2}$ for any $y_2 \in [2\epsilon,1]$.
\par

For the case where $\Sigma_y$ is not a subset of $B_1(0,\R^n)$, the trajectory $t \rightarrow X_{t}^{y,\epsilon}(\omega)$ with \eqref{thm2_eq10} must exit $\Sigma_y \cap B_1(0,\R^n)$ no later than the time when it exits $\Sigma_y$. The probability of exiting $\Sigma_y \cap B_1(0,\R^n)$ through $\partial(\Sigma_y \cap B_1(0,\R^n)) \backslash (\Gamma_y \cup \partial B_1(0,\R^n))$ is smaller than $\frac{4n}{\sqrt{2\pi}} \exp \big( -C_1 y_2^{-\alpha/2} \big)$. This concludes the proof of Claim 1.
\end{proof}
\end{proof}
\par

Finally, we prove \Cref{travel}.

\begin{proof}[Proof of \Cref{travel}]
For any $i \in \{1,3,4,\ldots,n\}$, define
\[
    \begin{split}
    \Gamma^{i+}_{\nu,\mu} &:= \big\{ (x_1,x_2,\ldots,x_n) \in \Sigma_{\nu,\mu} \, \big| \, 
        x_i = \nu x_2 \big\},  \\
    \Gamma^{i-}_{\nu,\mu} &:= \big\{ (x_1,x_2,\ldots,x_n) \in \Sigma_{\nu,\mu} \, \big| \, 
        x_i = -\nu x_2 \big\},  \\
    \Gamma'_{\nu,\mu} &:= \Big\{ (x_1,x_2,\ldots,x_n) \in \Sigma_{\nu,\mu} \, \Big| \, 
        x_2 = \frac{\mu}{2} \Big\},
    \end{split}
\]
then 
\[ \partial \Sigma_{\nu,\mu} = \Gamma_{\nu,\mu} \bigcup \Gamma'_{\nu,\mu} \bigcup \Big(\bigcup_{i}\Gamma^{i+}_{\nu,\mu}\Big) \bigcup \Big(\bigcup_{i}\Gamma^{i-}_{\nu,\mu}\Big). \]
\par

We prove $\rho(a) \in \Gamma_{\nu,\mu}$ by contradiction. Suppose $\rho(a) \in \Gamma^{i+}_{\nu,\mu}$ for some $i \in \{1,3,4,\ldots,n\}$ and we make the following claim. \par
\vspace{2mm}
\noindent \textbf{Claim 1:} For any $x \in \Sigma_{\nu,\mu}$ with $x_i \geq 0$, we have
\begin{equation} \label{travel_eq6} u(x) \cdot n_{i+} \geq 0, \end{equation}
where $n_{i+}$ is the outer normal vector of $\Sigma_{\nu,\mu}$ on the boundary component $\Gamma^{i+}_{\nu,\mu}$. \par
\vspace{2mm}

Because $\rho(a) \in \Gamma^{i+}_{\nu,\mu}$ and $\rho(0) \in S_{\nu,\mu,l}$, there exists $t_0 \in (0,a)$, such that 
\begin{equation} \label{travel_eq7}
    \begin{split}
        n_{i+} \cdot (\rho(a)-\rho(t_0)) &\geq \frac{(1-l)\nu\mu}{3}, \\
        \rho(t) \cdot \mathbf{e}_i &\geq 0 \quad \text{for any } t \in [t_0,a].
    \end{split}
\end{equation}
We multiply \eqref{travel_eq2} with $n_{i+}$
\[ n_{i+} \cdot (\rho(a)-\rho(t_0)) = -C\int_{t_0}^{a} n_{i+} \cdot u ( \rho(s) ) ds + n_{i+} \cdot (b(a)-b(t_0)), \]
then by \eqref{travel_eq7}, we can deduce
\[ \frac{(1-l)\nu\mu}{3} + C\int_{t_0}^{a} n_{i+} \cdot u ( \rho(s) ) ds \leq \frac{(1-l)\nu\mu}{4}, \]
which contradicts \eqref{travel_eq6} and $\rho(t) \cdot \mathbf{e}_i \geq 0$ for any $t \in [t_0,a]$.
\par

The possibility $\rho(a) \in \Gamma^{i-}_{\nu,\mu}$ or $\rho(a) \in \Gamma'_{\nu,\mu}$ can be excluded similarly.
\par

Now it suffices to prove Claim 1.
\begin{proof}[Proof of Claim 1]
The outer normal vector on the boundary component $\Gamma^{i+}_{\nu,\mu}$ is given by
\[ n_{i+}=\frac{1}{\sqrt{1+\nu^2}} \mathbf{e}_i + \frac{-\nu}{\sqrt{1+\nu^2}}\mathbf{e}_2 \]
and the velocity components $u(x) \cdot \mathbf{e}_1$ and $u(x) \cdot \mathbf{e}_j, j \geq 2$ are given by
\[
    \begin{split}
    u(x) \cdot \mathbf{e}_1 &= \frac{(r+z)^{n-3-\alpha}}{2r^{n-2}} - \frac{(r-z)^{n-3-\alpha}}{2r^{n-2}}, \\
    u(x) \cdot \mathbf{e}_j &= \Big( -\frac{(r+z)^{n-3-\alpha}}{2r^{n-2}} - \frac{(r-z)^{n-3-\alpha}}{2r^{n-2}} \Big) \cdot \frac{x_j}{r}, \quad 2 \leq j \leq n.
    \end{split}
\]
\par

If $i = 1$, for any $x \in \Sigma_{\nu,\mu}$ with $x_1>0$, we know $\cos \theta \in \big[(1+\nu^2(n-2))^{-\frac{1}{2}}, 1\big]$. Let $\zeta = \frac{z}{r} \in [0,\nu]$, then \eqref{travel_eq6} is equivalent to
\begin{equation} \label{travel_eq10}
    (1+\nu \cos \theta) (1+\zeta)^{n-3-\alpha} - (1-\nu \cos \theta) (1-\zeta)^{n-3-\alpha} \geq 0, \text{ for } \zeta \in [0,\nu].
\end{equation}
By elementary analysis, there exist $\alpha_0,\nu_0 > 0$ such that \eqref{travel_eq10} and hence \eqref{travel_eq6} hold for any $\nu \in (0,\nu_0]$ and $\alpha \in (0,\alpha_0]$. Indeed, we just need to look at the following function $\beta$ in $\zeta,\nu$ with parameter $\xi\in \big[(1+\nu^2(n-2))^{-\frac{1}{2}}, 1\big]$,
\[
    \beta(\zeta,\nu) := (1+\xi\nu) (1+\zeta)^{n-3-\alpha} - (1-\xi\nu) (1-\zeta)^{n-3-\alpha}.
\]
For $n \geq 4$, we have $\partial_\zeta \beta \geq 0$ and $\beta(0,\nu)\geq0$, so \eqref{travel_eq10} holds. For $n=3$, $\partial_\zeta \beta \leq 0$, it suffices to prove $\beta(\nu,\nu) \geq 0$. Now we look at the function $(1+\xi\nu) (1+\nu)^{-\alpha} - (1-\xi\nu) (1-\nu)^{-\alpha}$, its derivative in $\nu$ is positive for small $\nu>0$, so $\beta(\nu,\nu) \geq 0$ and thus \eqref{travel_eq10} holds.
\par

If $i \geq 3$, this is straightforward since \eqref{travel_eq6} is equivalent to $x_i \leq \nu x_2$.

\end{proof}
\end{proof}

\renewcommand{\bibname}{References}
\bibliography{bibliography}
\bibliographystyle{abbrv}

\end{document}